\documentclass[11pt, reqno]{amsart}
\usepackage[utf8]{inputenc}
\usepackage{amsmath}
\usepackage{amsthm}
\usepackage{amsfonts}
\usepackage{amssymb}
\usepackage{mathrsfs, mathtools, mathdots}
\usepackage{diagbox}
\usepackage{cancel}

\usepackage[top=30mm,bottom=30mm,left=27mm,right=27mm]{geometry}
\newtheorem{theorem}{Theorem}[section]

\newtheorem{corollary}{Corollary}[section]

\newtheorem*{theorem*}{Theorem}
\newtheorem*{remark*}{Remark}
\newtheorem*{problem*}{Problem}
\newtheorem*{conjecture*}{Conjecture}
\newtheorem{lemma}[theorem]{Lemma}

\newcommand{\R}{\mathbb{R}}

\title{On the distribution of $\phi(\sigma(n))$}
\author{Anup B. Dixit}
\author{Saunak Bhattacharjee}
\address{Indian Institute of Science Education and Research Tirupati, Srinivasapuram, Venkatagiri Road, Jangalapalli Village, Panguru (G.P), Yerpedu Mandal, Tirupati District, Andhra Pradesh, India 517619.}
\email{saunakbhattacharjee@students.iisertirupati.ac.in}
\address{Institute of Mathematical Sciences (HBNI), CIT Campus, Taramani, Chennai, Tamil Nadu, India 600113.}
\email{anupdixit@imsc.res.in}

\date{\today}

\begin{document}

\begin{abstract}
    Let $\phi(n)$ be the Euler totient function and $\sigma(n)$ denote the sum of divisors of $n$. In this note, we obtain explicit upper bounds on the number of positive integers $n\leq x$ such that $\phi(\sigma(n)) > cn$ for any $c>0$. This is a refinement of a result of Alaoglu and Erd\H{o}s.
\end{abstract}

\maketitle

\section{\bf{Introduction}}
\bigskip

For any positive integer $n$, let $\phi(n)$ be the Euler-totient function given by
\begin{equation*}
    \phi(n)= n \prod_{p|n}\left(1-\frac{1}{p}\right),
\end{equation*}
where $p$ runs over distinct primes dividing $n$. Let $\sigma(n)$ be the sum of divisors of $n$, which is given by
\begin{equation*}
    \sigma(n) = \sum_{d | n} d = n \prod_{p^k \| n} \left(\frac{1-p^{k+1}}{1-p}\right).
\end{equation*}
Here the notation $p^k \| n$ means that $p^k$ is the largest power of $p$ dividing $n$. In 1944, L. Alaoglu and P. Erd\H{o}s introduced the study of compositions of such arithmetic functions. In particular, they showed that for any real number $c>0$, 
\begin{align*}
    \#\{n\leq x: \phi\left(\sigma\left(n\right)\right)\geq cn \}= o(x) \qquad \text{ and } \qquad \#\{n\leq x: \sigma\left(\phi\left(n\right)\right)\leq cn \}= o(x).
\end{align*}
In \cite{luca}, F. Luca and C. Pomerance obtained finer results on the distribution of $\sigma(\phi(n))$. The objective of this paper is to study the distribution of $\phi(\sigma(n))$.\\

\noindent Denote by $\log_k $ the $k$-fold iterated logarithm $\log \log \cdots \log $ ($k$-times). We show that

\begin{theorem}\label{main-thm}
    For every $c>0$,
    \begin{equation*}
        \# \big\{ n\leq x: \phi\left(\sigma\left(n\right)\right)\geq cn \big\} \leq \frac{\pi^2 x}{6c\log_4 x} + O\left( \frac{x\log_3 x}{\left(\log x\right)^\frac{1}{\log_3 x}\log_4 x}\right),
    \end{equation*}
    where the implied constant only depends on $c$.
\end{theorem}
          
This implies that except for $\ll \frac{x}{\log\log\log\log x}$ integers less than $x$, $\phi(\sigma(n)) < cn$ for any $c>0$. It is possible to replace the constant $c$ above by a slowly decaying function. For a non-decreasing real function $f$, define
\begin{equation*}
    P_f(x) := \left\{n\leq x: \phi\left(\sigma\left(n\right)\right)\geq \frac{n}{f(n)} \right\}.
\end{equation*}
Then, we prove that
\begin{theorem}\label{thm-2}
    Suppose $f:\R^+ \to \R^+$ is a non-decreasing function satisfying
    \begin{equation*}
         f(x) = o\left(\log_4 x\right).
    \end{equation*}
    Then, 
    \begin{equation*}
        |P_f(x)| = O\left( \frac{xf(x)}{\log_4 x} +\frac{x\log_3 x}{\left(\log x\right)^\frac{1}{\log_3 x}\log_4 x}\right)= o(x)
    \end{equation*}
     as $x\to \infty$. In other words, for almost all positive integers $n$, $\phi(\sigma(n)) < \frac{n}{f(n)}$.
\end{theorem}

Choosing $f(x) = \log_5 x$ in Theorem \ref{thm-2}, we obtain the following corollary, which is an improvement of the result of Alaoglu and Erd\H{o}s \cite{erdos}.

\begin{corollary}
    Except for $O\left(\frac{x \log_5 x}{\log_4 x }\right)$ positive integers $n\leq x$,
    \begin{equation*}
        \phi(\sigma(n)) \leq \frac{n}{\log_5 n}.
    \end{equation*}
\end{corollary}

\bigskip

\section{\bf Preliminaries}
\bigskip

A necessary component of our proof is to estimate the number of positive integers not greater than $x$, which do not have certain prime factors. Such an estimate requires an application of Brun's sieve. For our purpose, we invoke the following result by P. Pollack and C. Pomerance \cite[Lemma 3]{pollack}.

\begin{lemma}\label{sieve-lemma}
    Let $P$ be a set of primes and for $x>1$, let
    \begin{equation*}
        A(x)=\sum_{ \substack{p\leq x\\  p\in P}} \frac{1}{p}.
    \end{equation*}
    Then uniformly for all choices of $P$, the proportion of  $n\leq x$ free of prime factors from $P$ is $O\left(e^{-A(x)}\right)$.
\end{lemma}

We also recall the famous Siegel-Walfisz theorem (see \cite[Corollary 11.21]{Montgomery}).

\begin{lemma}[Siegel-Walfisz]\label{siegel}
    For $(a,q) = 1$, let  $\pi(x;q, a) $ denote the number of primes $p\leq x$ such that $p\equiv a (\bmod \, q)$. Let $A>0$ be given. If $q \leq (\log x)^A$, then
    \begin{equation*}
        \pi(x; q,a) = \frac{li(x)}{\phi(q)} + O\left(x \exp(-c \sqrt{\log x})\right),
    \end{equation*}
    where the implied constant only depends on $A$ and $li(x) := \int_2^{x} \frac{1}{\log t} \, dt$.
\end{lemma}

For any prime $p$, define

\begin{equation*}
    S_p\left(x\right) := \#\{n\leq x : p\nmid \sigma\left(n\right)\}.
\end{equation*}

The main ingredient in the proof of Theorem \ref{main-thm}, which is also interesting in its own right, is an upper bound for $S_p(x)$.

\begin{lemma}\label{main-lemma}
    For any prime $p$ and $x\geq e^p$
\begin{equation*}
    S_p(x) = O\left(x \left(\frac{\log\log x}{\log x}\right)^{\frac{1}{p-1}}\right),
\end{equation*}
where the implied constant is absolute.
\end{lemma}

\begin{proof}
    Note that for any prime $q\equiv -1 \bmod  p$, all $n$ such that $q \|n$ satisfy $p|\sigma(n)$. Thus, to obtain an upper bound for $S_p(x)$, it suffices to estimate the number of $n\leq x$ such that either $q\nmid n $ or $q^2 \mid n$ for a subset of primes $q \equiv -1 (\bmod \, p)$. By Lemma \ref{siegel}, for $x> e^p$, we have
    \begin{equation*}
        \pi(x;p, -1) = \frac{x}{(p-1) \log x} + O\left(\frac{x}{(\log x)^2}\right),
    \end{equation*}
    where the implied constant is absolute. Now suppose $x$ is sufficiently large such that $\log x > e^{p}$. Applying partial summation, we obtain
    \begin{align*}
        \sum_{\substack{\log x\, < \,q \,<\, x \\ q\equiv -1 (\bmod\, p)}} \frac{1}{q} &= \frac{\pi(x; p, -1)}{ x}  -\frac{\pi(\log x; p, -1)}{\log x} + \int_{\log x}^{x} \frac{\pi(t; p, -1)}{t^2}\, dt\\
        & = \frac{1}{p-1} \, \int_{\log x}^x \frac{1}{t\log t}\, dt + O\left(\frac{1}{\log_2 x}\right)\\
        & = \frac{1}{p-1} \left(\log_2 x - \log_3 x\right) + O\left(\frac{1}{\log_2 x}\right).
    \end{align*}

\noindent    Now, applying Lemma \ref{sieve-lemma} with $P$ being the set of primes $q\equiv -1 \bmod  p$ and $\log x < q < x$, we obtain the number of $n\leq x$ free of prime factors from $P$ is
    $$
        O\left(x \left(\frac{\log_2 x}{\log x}\right)^{\frac{1}{p-1}}\right).
    $$
    Since
    \begin{align*}
    \#\{n\leq x : q^2 | n \text{ for prime } q\equiv -1 \bmod p \text{ and } \log x < q < x\} & \ll x \sum_{\log x < q < x } \frac{1}{q^2} 
    \ll \frac{x}{\log x},
    \end{align*}
    we have the lemma.
    
\end{proof}

\medskip

\section{\textbf{Proof of Theorems \ref{main-thm} and \ref{thm-2}}}
\medskip
Note that

\begin{equation*}               \phi\left(\sigma\left(n\right)\right)=\sigma\left(n\right)\prod_{p \mid \sigma\left(n\right)} \left(1 - \frac{1}{p}\right)
\end{equation*}
Denote by $P(y):= \prod\limits_{p\leq y} p$, the product of all primes $\leq y$. If $P(y) | \sigma(n)$, then
\begin{align*}
    \phi(\sigma(n)) &= \sigma(n) \prod_{p|\sigma(n)} \left(1-\frac{1}{p}\right)\\
    & \leq \sigma(n) \prod_{p\leq y} \left(1- \frac{1}{p}\right) < \frac{\sigma(n)}{\log y},
\end{align*}
where the last inequality follows from Merten's theorem (see \cite[Theorem 2.7 (e)]{Montgomery}), namely
\begin{equation*}
    \prod_{p\leq y} \left( 1- \frac{1}{p}\right) < \frac{1}{\log y}.
\end{equation*}
Thus, for any $c>0$, $\phi(\sigma(n))< cn$ holds if $P(y) | \sigma(n)$, $\sigma(n)< \delta n$ and $(\log y)^{-1} \leq c/\delta$.
We know that (see \cite[Theorem 3.4]{apostol})
\begin{equation*}
\sum_{n \leq x} \sigma(n) = \frac{\pi^2}{12}x^2 + O\left(x\log x\right).
\end{equation*}
Using partial summation, we get
\begin{equation*}
    \sum_{n \leq x} \frac{\sigma(n)}{n} = \frac{\pi^2}{6}x + O\left(\log^2 x \right).
\end{equation*}
Hence,
\begin{align*}
    \#\{n\leq x : \sigma\left(n\right)\geq \delta n\}& = \sum_{\substack{n \leq x \\ \sigma(n) \geq \delta n}} 1 \leq  \frac{1}{\delta} \sum_{n \leq x} \frac{\sigma(n)}{n}\\
    & = \frac{\pi^2}{6\delta}x + O\left(\frac{\log^2 x}{\delta} \right).
\end{align*}
Therefore,
\begin{equation}\label{eqn-1}
    \#\{n\leq x : \sigma\left(n\right) < \delta n\} \geq x \left( 1 - \frac{\pi^2}{6\delta}\right) + O\left(\frac{\log^2 x}{\delta}\right). 
\end{equation}
From Lemma \ref{main-lemma}, we also have
\begin{align*}
    \#\{n\leq x: P(y) \nmid \sigma(n)\} & \leq \sum_{p\leq y} |S_p(x)| \\
    & = O\left( x\left(\frac{\log_2 x}{\log x}\right)^\frac{1}{y} \frac{y}{\log y}\right).
\end{align*}
Hence,
\begin{equation}\label{eqn-2}
    \#\{n\leq x: P(y) \mid \sigma(n)\} \geq x \left(1 - O\left( \left(\frac{\log_2 x}{\log x}\right)^\frac{1}{y} \frac{y}{\log y}\right)\right).
\end{equation}
Choosing
\begin{equation*}
    y=\log_3 x \,\,\, \text{ and }\,\,\,  \delta = c \log_4 x 
\end{equation*}
in \eqref{eqn-1} and \eqref{eqn-2}, we obtain
\begin{equation*}
    \#\{n\leq x : \phi(\sigma(n)) < cn\} \geq x - \frac{\pi^2 x}{6c\log_4 x} + O\left( \frac{x\log_3 x}{\left(\log x\right)^\frac{1}{\log_3 x}\log_4 x}\right).
\end{equation*}
Hence,
\begin{equation*}
     \# \big\{ n\leq x: \phi\left(\sigma\left(n\right)\right)\geq cn \big\}\leq \frac{\pi^2 x}{6c\log_4 x} + O\left( \frac{x\log_3 x}{\left(\log x\right)^\frac{1}{\log_3 x}\log_4 x}\right),
\end{equation*}
which proves Theorem \ref{main-thm}.\\

\medskip

The proof of Theorem \ref{thm-2} follows the exact same method as above, with the choices
\begin{equation*}
    y=\log_3 x \,\,\, \text{ and }\,\,\,  \delta =  \frac{\log_4 x}{f(x)}
\end{equation*}
in \eqref{eqn-1} and \eqref{eqn-2}. This gives
\begin{equation*}
    \#\left\{n\leq x : \phi(\sigma(n)) < \frac{n}{f(n)}\right\} \geq x - O\left(\frac{xf(x)}{ \log_4 x}+\frac{x\log_3 x}{\left(\log x\right)^\frac{1}{\log_3 x}\log_4 x}\right).
\end{equation*}
This proves Theorem \ref{thm-2}.

\bigskip

\section{\bf Concluding remarks}
\medskip
The study of composition of multiplicative arithmetic functions seems to be a difficult theme in general. This has also received scant attention, except for a very few instances such as \cite{erdos} and \cite{pollack}. For example, it is not clear if $\phi(\sigma(n))$ has a normal order. It would be desirable to develop a unified theory for such functions and perhaps construct families of multiplicative functions whose compositions have a finer distribution.

\bigskip

\section{\bf Acknowledgements}
We thank Prof. Jean-Marc Deshouillers for several fruitful discussions. We also thank the referee for comments on an earlier version of this paper.

\end{document}